\documentclass{article}
\usepackage{amsmath}
\usepackage{amsfonts}
\usepackage{amsthm}
\usepackage{geometry}
\title{Tensor distributions with covariance tensor or correlation tensor}
\author{Yurii Yurchenko\thanks{Odessa Polytechnic State University, Institute of Computer Systems, Department of Applied Mathematics and Information Technology, Shevchenko av. 1, Odessa 65044, Ukraine}}
\date{August, 2021}
\begin{document}
\maketitle
\begin{abstract}
In this article, we define the matricization of a tensor and we present some properties of the matricization. After that, we define the determinant of a tensor and we present some properties of the determinant. We define the covariance tensor and we present some properties of the covariance tensor. In a similar way, we define the correlation tensor. We define the tensor normal distribution. In a similar way, we define the tensor elliptical distributions. We prove the equivalence of the tensor elliptical distribution representations.
\end{abstract}
\newtheorem{prop}{Proposition}[section]
\newtheorem{cor}{Corollary}[section]
\newtheorem{defi}{Definition}[section]
\newtheorem{property}{Property}[section]
\newtheorem{lemma}{Lemma}[section]
\newtheorem{theorem}{Theorem}[section]
\section{Introduction}
Let $\boldsymbol{x}$ is a k-dimensional random vector and $\boldsymbol{x}\sim\mathcal{N}_{n}\left(\boldsymbol{\mu},\mathbf{\Sigma}\right)$, then, by definition \cite{kotz}, the probability density function of the random vector is given by
\[
f(\boldsymbol{x})=\frac{\operatorname{exp}\left(-\frac{1}{2}\left(\boldsymbol{x}-\boldsymbol{\mu}\right)^{\mathrm{T}}\mathbf{\Sigma}^{-1}\left(\boldsymbol{x}-\boldsymbol{\mu}\right)\right)}{(2\pi)^{n/2}\det\left(\mathbf{\Sigma}\right)^{1/2}}.
\]
Let's rewrite this formula in a different form:
\[
f(\boldsymbol{x})=\frac{\operatorname{exp}\left(-\frac{1}{2}\left(\boldsymbol{x}-\boldsymbol{\mu}\right):\left(\mathbf{\Sigma}^{-1}\right):\left(\boldsymbol{x}-\boldsymbol{\mu}\right)\right)}{(2\pi)^{n/2}\det\left(\mathbf{\Sigma}\right)^{1/2}},
\]
where $:$ is the double dot product of tensors, i.e.
\[
\left(\boldsymbol{x}-\boldsymbol{\mu}\right):\left(\mathbf{\Sigma}^{-1}\right):\left(\boldsymbol{x}-\boldsymbol{\mu}\right)=\left(\boldsymbol{x}-\boldsymbol{\mu}\right)_{i}\left(\mathbf{\Sigma}^{-1}\right)^{i}_{j}\left(\boldsymbol{x}-\boldsymbol{\mu}\right)^{j}=\sum_{i,j}\left(\boldsymbol{x}-\boldsymbol{\mu}\right)_{i}\left(\mathbf{\Sigma}^{-1}\right)_{i,j}\left(\boldsymbol{x}-\boldsymbol{\mu}\right)_{j},
\]
where we use the Einstein summation convention \cite{its}.

If $\boldsymbol{x}\sim\mathcal{N}_{n}\left(\boldsymbol{\mu},\mathbf{\Sigma}\right)$, then, by property \cite{kotz}, the covariance matrix $\mathbf{K}_{\boldsymbol{xx}}$ of the random vector is equal to $\mathbf{\Sigma}$.

By definition \cite{kotz}, the covariance matrix of a random vector is given by
\[
\mathbf{K}_{\boldsymbol{xx}}=\mathbb{E}\left[\left(\boldsymbol{x}-\mathbb{E}\left[\boldsymbol{x}\right]\right)\left(\boldsymbol{x}-\mathbb{E}\left[\boldsymbol{x}\right]\right)^{\mathrm{T}}\right].
\]
In other words,
\[
\left(\mathbf{K}_{\boldsymbol{xx}}\right)_{i,j}=\mathbb{E}\left[\left(\boldsymbol{x}_i-\mathbb{E}\left[\boldsymbol{x}_i\right]\right)\left(\boldsymbol{x}_j-\mathbb{E}\left[\boldsymbol{x}_j\right]\right)^{\mathrm{T}}\right]=\operatorname{cov}\left(\boldsymbol{x}_i,\boldsymbol{x}_j\right).
\]

The covariance matrix is a generalization of the variance for random vectors. Also, the covariance matrix $\mathbf{K}_{\boldsymbol{xx}}$ is the order-2 tensor, with dimensional lengths $n \times n$, if $\boldsymbol{x}$ is a n-dimensional random vector, i.e. a random order-1 tensor. 
\section{Tensor matricization, inverse tensor and tensor determinant}
Let $\operatorname{mat}\left(\boldsymbol{\mathcal{X}}\right)$ denote the matricization of tensor $\boldsymbol{\mathcal{X}}$.
\begin{defi}
Let $\boldsymbol{\mathcal{X}}=\left(\boldsymbol{\mathcal{X}}_{i_1,\ldots,i_D,i_{D+1},\ldots,i_{2D}}\right)$ is an order-2D tensor, with dimensional lengths $n_1 \times n_2 \times\ldots\times n_D\times n_1 \times n_2 \times\ldots\times n_D=\mathbf{n}\times \mathbf{n}$, then the matricization of the tensor is given by
\[
\operatorname{mat}\left(\boldsymbol{\mathcal{X}}\right)=\begin{pmatrix}
\boldsymbol{\mathcal{X}}_{1,\ldots,1,1\ldots,1}&\boldsymbol{\mathcal{X}}_{1,\ldots,1,2\ldots,1}&\cdots&\boldsymbol{\mathcal{X}}_{1,\ldots,1,1,2\ldots,1}&\cdots&\boldsymbol{\mathcal{X}}_{1,\ldots,1,n_{1},\ldots,n_{D}}\\
\boldsymbol{\mathcal{X}}_{2,\ldots,1,1\ldots,1}&\boldsymbol{\mathcal{X}}_{2,\ldots,1,2\ldots,1}&\cdots&\boldsymbol{\mathcal{X}}_{2,\ldots,1,1,2\ldots,1}&\cdots&\boldsymbol{\mathcal{X}}_{2,\ldots,1,n_{1},\ldots,n_{D}}\\
\vdots&\vdots&\ddots&\vdots&\ddots&\vdots\\
\boldsymbol{\mathcal{X}}_{1,2,\ldots,1,1,\ldots,1}&\boldsymbol{\mathcal{X}}_{1,2,\ldots,1,2\ldots,1}&\cdots&\boldsymbol{\mathcal{X}}_{1,2,\ldots,1,1,2,\ldots,1}&\cdots&\boldsymbol{\mathcal{X}}_{1,2,\ldots,1,n_{1},\ldots,n_{D}}\\
\vdots&\vdots&\ddots&\vdots&\ddots&\vdots\\
\boldsymbol{\mathcal{X}}_{n_{1},\ldots,n_{D},1,\ldots,1}&\boldsymbol{\mathcal{X}}_{n_{1},\ldots,n_{D},1,2,\ldots,1}&\cdots&\boldsymbol{\mathcal{X}}_{n_{1},\ldots,n_{D},1,2,\ldots,1}&\cdots&\boldsymbol{\mathcal{X}}_{n_{1},\ldots,n_{D},n_{1},\ldots,n_{D}}\\
\end{pmatrix}.
\]
\end{defi}
The matricization of a tensor has the following properties. The proof of the properties is trivial, so we omit it.
\begin{property}
\label{zeromat}
$\operatorname{mat}\left(\boldsymbol{\mathcal{O}}\right)=\mathbf{O}$, where $\boldsymbol{\mathcal{O}}$ is the zero tensor, i.e  $\boldsymbol{\mathcal{O}}_{i_1,\ldots,i_D,i_{D+1},\ldots,i_{2D}}=0$, and $\mathbf{O}$ is the zero matrix.
\end{property}
\begin{property}
\label{identitymat}
$\operatorname{mat}\left(\boldsymbol{\mathcal{I}}\right)=\mathbf{I}$, where $\mathbf{I}$ is the identity matrix and $\boldsymbol{\mathcal{I}}$ is the identity tensor, i.e  
\[
\boldsymbol{\mathcal{I}}_{i_1,\ldots,i_D,i_{D+1},\ldots,i_{2D}}=\begin{cases} 1, & i_1=\ldots=i_D=i_{D+1}=\ldots=i_{2D}\\0, & \text{otherwise}\end{cases}.
\]
\end{property}
\begin{property}
\label{scalmat}
$\operatorname{mat}\left(\lambda\boldsymbol{\mathcal{X}}\right)=\lambda\operatorname{mat}\left(\boldsymbol{\mathcal{X}}\right)$, where $\lambda$ is a scalar.
\end{property}
\begin{property}
$\operatorname{mat}\left(\boldsymbol{\mathcal{X}}+\boldsymbol{\mathcal{Y}}\right)=\operatorname{mat}\left(\boldsymbol{\mathcal{X}}\right)+\operatorname{mat}\left(\boldsymbol{\mathcal{Y}}\right)$.
\end{property}
\begin{property}
\label{tran}
Let $\boldsymbol{\mathcal{X}}=\left(\boldsymbol{\mathcal{X}}_{i_1,\ldots,i_D,i_{D+1},\ldots,i_{2D}}\right)$ is an order-2D tensor, with dimensional lengths $n_1 \times n_2 \times\ldots\times n_D\times n_1 \times n_2 \times\ldots\times n_D=\mathbf{n}\times \mathbf{n}$ and $\boldsymbol{\mathcal{X}}^{\mathrm{T}}=\left(\boldsymbol{\mathcal{X}}_{i_{D+1},\ldots,i_{2D},i_1,\ldots,i_D}\right)$, then
\[
\operatorname{mat}\left(\boldsymbol{\mathcal{X}}^{\mathrm{T}}\right)=\operatorname{mat}\left(\boldsymbol{\mathcal{X}}\right)^{\mathrm{T}}.
\]
\end{property}
\begin{property}
\label{prodmat}
Let $\boldsymbol{\mathcal{X}}=\left(\boldsymbol{\mathcal{X}}_{i_1,\ldots,i_D,j_1,\ldots,j_D}\right)$ is an order-2D tensor, with dimensional lengths $n_1 \times n_2 \times\ldots\times n_D\times n_1 \times n_2 \times\ldots\times n_D=\mathbf{n}\times \mathbf{n}$, and let $\boldsymbol{\mathcal{Y}}=\left(\boldsymbol{\mathcal{Y}}_{j_1,\ldots,j_D,k_1,\ldots,k_D}\right)$ is an order-2D tensor, with dimensional lengths $n_1 \times n_2 \times\ldots\times n_D\times n_1 \times n_2 \times\ldots\times n_D=\mathbf{n}\times\mathbf{n}$, then 
\[
\operatorname{mat}\left(\boldsymbol{\mathcal{X}}^{i_1,\ldots,i_D}_{j_1,\ldots,j_D}\boldsymbol{\mathcal{Y}}^{j_1,\ldots,j_D}_{k_1,\ldots,k_D}\right)=\operatorname{mat}\left(\boldsymbol{\mathcal{X}}\right)\operatorname{mat}\left(\boldsymbol{\mathcal{Y}}\right).
\]
\end{property}
\begin{defi}
\label{defimat-1}
Let $\boldsymbol{\mathcal{X}}=\left(\boldsymbol{\mathcal{X}}_{i_1,\ldots,i_D,j_1,\ldots,j_D}\right)$ is an order-2D tensor, with dimensional lengths $n_1 \times n_2 \times\ldots\times n_D\times n_1 \times n_2 \times\ldots\times n_D=\mathbf{n}\times \mathbf{n}$, then the inverse tensor of the tensor $\boldsymbol{\mathcal{X}}$ is given by
\[
\boldsymbol{\mathcal{X}}^{i_1,\ldots,i_D}_{j_1,\ldots,j_D}\left(\boldsymbol{\mathcal{X}}^{-1}\right)_{k_1,\ldots,k_D}^{j_1,\ldots,j_D}=\left(\boldsymbol{\mathcal{X}}^{-1}\right)^{i_1,\ldots,i_D}_{j_1,\ldots,j_D}\boldsymbol{\mathcal{X}}_{k_1,\ldots,k_D}^{j_1,\ldots,j_D}=\boldsymbol{\mathcal{I}},
\]
where $\boldsymbol{\mathcal{I}}$ is the identity tensor.
\end{defi}
\begin{property}
\label{mat-1}
Let $\boldsymbol{\mathcal{X}}=\left(\boldsymbol{\mathcal{X}}_{i_1,\ldots,i_D,j_1,\ldots,j_D}\right)$ is an order-2D tensor, with dimensional lengths $n_1 \times n_2 \times\ldots\times n_D\times n_1 \times n_2 \times\ldots\times n_D=\mathbf{n}\times \mathbf{n}$, then 
\[
\operatorname{mat}\left(\boldsymbol{\mathcal{X}}^{-1}\right)=\operatorname{mat}\left(\boldsymbol{\mathcal{X}}\right)^{-1}.
\]
\end{property}
\begin{proof}
By Definition \ref{defimat-1},
\[
\boldsymbol{\mathcal{X}}^{i_1,\ldots,i_D}_{j_1,\ldots,j_D}\left(\boldsymbol{\mathcal{X}}^{-1}\right)_{k_1,\ldots,k_D}^{j_1,\ldots,j_D}=\left(\boldsymbol{\mathcal{X}}^{-1}\right)^{i_1,\ldots,i_D}_{j_1,\ldots,j_D}\boldsymbol{\mathcal{X}}_{k_1,\ldots,k_D}^{j_1,\ldots,j_D}=\boldsymbol{\mathcal{I}},
\]
By Property \ref{prodmat},
\[
\operatorname{mat}\left(\boldsymbol{\mathcal{X}}\right)\operatorname{mat}\left(\boldsymbol{\mathcal{X}}^{-1}\right)=\operatorname{mat}\left(\boldsymbol{\mathcal{X}}^{-1}\right)\operatorname{mat}\left(\boldsymbol{\mathcal{X}}\right)=\operatorname{mat}\left(\boldsymbol{\mathcal{I}}\right).
\]
By Property \ref{identitymat},
\[
\operatorname{mat}\left(\boldsymbol{\mathcal{X}}\right)\operatorname{mat}\left(\boldsymbol{\mathcal{X}}^{-1}\right)=\operatorname{mat}\left(\boldsymbol{\mathcal{X}}^{-1}\right)\operatorname{mat}\left(\boldsymbol{\mathcal{X}}\right)=\mathbf{I}.
\]
Thus,
\[
\operatorname{mat}\left(\boldsymbol{\mathcal{X}}^{-1}\right)=\operatorname{mat}\left(\boldsymbol{\mathcal{X}}\right)^{-1}.
\]
\end{proof}
\begin{defi}
\label{det}
Let $\boldsymbol{\mathcal{X}}=\left(\boldsymbol{\mathcal{X}}_{i_1,\ldots,i_D,i_{D+1},\ldots,i_{2D}}\right)$ is an order-2D tensor, with dimensional lengths $n_1 \times n_2 \times\ldots\times n_D\times n_1 \times n_2 \times\ldots\times n_D=\mathbf{n}\times \mathbf{n}$, then the determinant of the tensor is given by
\[
\det\left(\boldsymbol{\mathcal{X}}\right)=\det\left(\operatorname{mat}\left(\boldsymbol{\mathcal{X}}\right)\right).
\]
\end{defi}
The determinant of a tensor has the following properties.
\begin{property}
$\det\left(\boldsymbol{\mathcal{O}}\right)=0$, where $\boldsymbol{\mathcal{O}}$ is the zero tensor, i.e  $\boldsymbol{\mathcal{O}}_{i_1,\ldots,i_D,i_{D+1},\ldots,i_{2D}}=0.$
\end{property}
\begin{proof}
By Definition \ref{det},
\[
\det\left(\boldsymbol{\mathcal{O}}\right)=\det\left(\operatorname{mat}\left(\boldsymbol{\mathcal{O}}\right)\right).
\]
By Property \ref{zeromat},
\[
\det\left(\boldsymbol{\mathcal{O}}\right)=\det\left(\mathbf{O}\right).
\]
Thus, $\det\left(\boldsymbol{\mathcal{O}}\right)=0.$
\end{proof}
\begin{property}
$\det\left(\boldsymbol{\mathcal{I}}\right)=1$, where $\boldsymbol{\mathcal{I}}$ is the identity tensor, i.e  
\[
\boldsymbol{\mathcal{I}}_{i_1,\ldots,i_D,i_{D+1},\ldots,i_{2D}}=\begin{cases} 1, & i_1=\ldots=i_D=i_{D+1}=\ldots=i_{2D}\\0, & \text{otherwise}\end{cases}.
\]
\end{property}
\begin{proof}
By Definition \ref{det},
\[
\det\left(\boldsymbol{\mathcal{I}}\right)=\det\left(\operatorname{mat}\left(\boldsymbol{\mathcal{I}}\right)\right).
\]
By Property \ref{identitymat},
\[
\det\left(\boldsymbol{\mathcal{I}}\right)=\det\left(\mathbf{I}\right).
\]
Thus, $\det\left(\boldsymbol{\mathcal{I}}\right)=1.$
\end{proof}
\begin{property}
Let $\boldsymbol{\mathcal{X}}=\left(\boldsymbol{\mathcal{X}}_{i_1,\ldots,i_D,i_{D+1},\ldots,i_{2D}}\right)$ is an order-2D tensor, with dimensional lengths $n_1 \times n_2 \times\ldots\times n_D\times n_1 \times n_2 \times\ldots\times n_D=\mathbf{n}\times \mathbf{n}$, then the determinant of the tensor is given by
\[
\det\left(\lambda\boldsymbol{\mathcal{X}}\right)=\lambda^{n^{*}}\det\left(\boldsymbol{\mathcal{X}}\right),
\]
where $\lambda$ is a scalar and $n^{*}=\prod_{i=1}^{D}n_i$.
\end{property}
\begin{proof}
By Definition \ref{det},
\[
\det\left(\lambda\boldsymbol{\mathcal{X}}\right)=\det\left(\operatorname{mat}\left(\lambda\boldsymbol{\mathcal{X}}\right)\right).
\]
By Property \ref{scalmat},
\[
\det\left(\lambda\boldsymbol{\mathcal{X}}\right)=\det\left(\lambda\operatorname{mat}\left(\boldsymbol{\mathcal{X}}\right)\right),
\]
where $\operatorname{mat}\left(\boldsymbol{\mathcal{X}}\right)$ is a square $n^{*}\times n^{*}$ matrix.

Therefore,
\[
\det\left(\lambda\boldsymbol{\mathcal{X}}\right)=\lambda^{n^{*}}\det\left(\operatorname{mat}\left(\boldsymbol{\mathcal{X}}\right)\right).
\]
Thus,
\[
\det\left(\lambda\boldsymbol{\mathcal{X}}\right)=\lambda^{n^{*}}\det\left(\boldsymbol{\mathcal{X}}\right).
\]
\end{proof}
\begin{property}
Let $\boldsymbol{\mathcal{X}}=\left(\boldsymbol{\mathcal{X}}_{i_1,\ldots,i_D,i_{D+1},\ldots,i_{2D}}\right)$ is an order-2D tensor, with dimensional lengths $n_1 \times n_2 \times\ldots\times n_D\times n_1 \times n_2 \times\ldots\times n_D=\mathbf{n}\times \mathbf{n}$ and $\boldsymbol{\mathcal{X}}^{\mathrm{T}}=\left(\boldsymbol{\mathcal{X}}_{i_{D+1},\ldots,i_{2D},i_1,\ldots,i_D}\right)$, then
\[
\det\left(\boldsymbol{\mathcal{X}}^{\mathrm{T}}\right)=\det\left(\boldsymbol{\mathcal{X}}\right).
\]
\end{property}
\begin{proof}
By Definition \ref{det},
\[
\det\left(\boldsymbol{\mathcal{X}}^{\mathrm{T}}\right)=\det\left(\operatorname{mat}\left(\boldsymbol{\mathcal{X}}^{\mathrm{T}}\right)\right).
\]
By Property \ref{tran},
\[
\det\left(\boldsymbol{\mathcal{X}}^{\mathrm{T}}\right)=\det\left(\operatorname{mat}\left(\boldsymbol{\mathcal{X}}\right)^{\mathrm{T}}\right).
\]
Thus,
\[
\det\left(\boldsymbol{\mathcal{X}}^{\mathrm{T}}\right)=\det\left(\operatorname{mat}\left(\boldsymbol{\mathcal{X}}\right)\right)=\det\left(\boldsymbol{\mathcal{X}}\right).
\]
\end{proof}
\begin{property}
Let $\boldsymbol{\mathcal{X}}=\left(\boldsymbol{\mathcal{X}}_{i_1,\ldots,i_D,j_1,\ldots,j_D}\right)$ is an order-2D tensor, with dimensional lengths $n_1 \times n_2 \times\ldots\times n_D\times n_1 \times n_2 \times\ldots\times n_D=\mathbf{n}\times \mathbf{n}$, and let $\boldsymbol{\mathcal{Y}}=\left(\boldsymbol{\mathcal{Y}}_{j_1,\ldots,j_D,k_1,\ldots,k_D}\right)$ is an order-2D tensor, with dimensional lengths $n_1 \times n_2 \times\ldots\times n_D\times n_1 \times n_2 \times\ldots\times n_D=\mathbf{n}\times\mathbf{n}$, then 
\[
\det\left(\boldsymbol{\mathcal{X}}^{i_1,\ldots,i_D}_{j_1,\ldots,j_D}\boldsymbol{\mathcal{Y}}^{j_1,\ldots,j_D}_{k_1,\ldots,k_D}\right)=\det\left(\boldsymbol{\mathcal{X}}\right)\det\left(\boldsymbol{\mathcal{Y}}\right).
\]
\end{property}
\begin{proof}
By Property \ref{prodmat},
\[
\operatorname{mat}\left(\boldsymbol{\mathcal{X}}^{i_1,\ldots,i_D}_{j_1,\ldots,j_D}\boldsymbol{\mathcal{Y}}^{j_1,\ldots,j_D}_{k_1,\ldots,k_D}\right)=\operatorname{mat}\left(\boldsymbol{\mathcal{X}}\right)\operatorname{mat}\left(\boldsymbol{\mathcal{Y}}\right).
\]
Therefore,
\[
\det\left(\operatorname{mat}\left(\boldsymbol{\mathcal{X}}^{i_1,\ldots,i_D}_{j_1,\ldots,j_D}\boldsymbol{\mathcal{Y}}^{j_1,\ldots,j_D}_{k_1,\ldots,k_D}\right)\right)=\det\left(\operatorname{mat}\left(\boldsymbol{\mathcal{X}}\right)\right)\det\left(\operatorname{mat}\left(\boldsymbol{\mathcal{Y}}\right)\right).
\]
By Definition \ref{det},
\[
\det\left(\boldsymbol{\mathcal{X}}^{i_1,\ldots,i_D}_{j_1,\ldots,j_D}\boldsymbol{\mathcal{Y}}^{j_1,\ldots,j_D}_{k_1,\ldots,k_D}\right)=\det\left(\boldsymbol{\mathcal{X}}\right)\det\left(\boldsymbol{\mathcal{Y}}\right).
\]
\end{proof}
\begin{property}
Let $\boldsymbol{\mathcal{X}}=\left(\boldsymbol{\mathcal{X}}_{i_1,\ldots,i_D,j_1,\ldots,j_D}\right)$ is an order-2D tensor, with dimensional lengths $n_1 \times n_2 \times\ldots\times n_D\times n_1 \times n_2 \times\ldots\times n_D=\mathbf{n}\times \mathbf{n}$, then 
\[
\det\left(\boldsymbol{\mathcal{X}}^{-1}\right)=\det\left(\boldsymbol{\mathcal{X}}\right)^{-1}.
\]
\end{property}
\begin{proof}
By Property \ref{mat-1},
\[
\operatorname{mat}\left(\boldsymbol{\mathcal{X}}^{-1}\right)=\operatorname{mat}\left(\boldsymbol{\mathcal{X}}\right)^{-1}.
\]
Therefore,
\[
\det\left(\operatorname{mat}\left(\boldsymbol{\mathcal{X}}^{-1}\right)\right)=\det\left(\operatorname{mat}\left(\boldsymbol{\mathcal{X}}\right)\right)^{-1}.
\]
By Definition \ref{det},
\[
\det\left(\boldsymbol{\mathcal{X}}^{-1}\right)=\det\left(\boldsymbol{\mathcal{X}}\right)^{-1}.
\]
\end{proof}
\section{Covariance tensor}
The cross-covariance tensor of random tensors $\boldsymbol{\mathcal{X}}$ and $\boldsymbol{\mathcal{Y}}$ will be denoted by $\boldsymbol{\mathcal{K}}_{\boldsymbol{\mathcal{X}}\boldsymbol{\mathcal{Y}}}$.
\begin{defi}
Let $\boldsymbol{\mathcal{X}}$ is a random order-D tensor with dimensional lengths $n_1 \times n_2 \times\ldots\times n_D=\mathbf{n}$ and $\boldsymbol{\mathcal{Y}}$ is a random order-D tensor with dimensional lengths $m_1 \times m_2 \times\ldots\times m_D=\mathbf{m}$, then the cross-covariance tensor is given by
\[
\boldsymbol{\mathcal{K}}_{\boldsymbol{\mathcal{X}}\boldsymbol{\mathcal{Y}}}=\mathbb{E}\left[\left(\boldsymbol{\mathcal{X}}-\mathbb{E}\left[\boldsymbol{\mathcal{X}}\right]\right)\otimes_{outer}\left(\boldsymbol{\mathcal{Y}}-\mathbb{E}\left[\boldsymbol{\mathcal{Y}}\right]\right)\right],
\]
where $\otimes_{outer}$ is the outer product of tensors. In other words,
\[
\left(\boldsymbol{\mathcal{K}}_{\boldsymbol{\mathcal{X}}\boldsymbol{\mathcal{Y}}}\right)_{i_1,\ldots,i_D,i_{D+1},\ldots,i_{2D}}=\operatorname{cov}\left(\boldsymbol{\mathcal{X}}_{i_1,\ldots,i_D},\boldsymbol{\mathcal{Y}}_{i_{D+1},\ldots,i_{2D}}\right).
\]
\end{defi}
 Thus, if $\boldsymbol{\mathcal{X}}$ is a random order-D tensor with dimensional lengths $n_1 \times n_2 \times\ldots\times n_D=\mathbf{n}$ and $\boldsymbol{\mathcal{Y}}$ is a random order-D tensor with dimensional lengths $m_1 \times m_2 \times\ldots\times m_D=\mathbf{m}$, then the cross-covariance tensor is the order-2D tensor,  with dimensional lengths $\mathbf{n}\times\mathbf{m}$.

The cross-covariance tensor of random tensors has the following properties. The proof of the properties is trivial, so we omit it.
\begin{property}
$\left(\boldsymbol{\mathcal{K}}_{\boldsymbol{\mathcal{X}},\boldsymbol{\mathcal{Y}}}\right)_{i_1,\ldots,i_D,i_{D+1},\ldots,i_{2D}}=\left(\boldsymbol{\mathcal{K}}_{\boldsymbol{\mathcal{Y}},\boldsymbol{\mathcal{X}}}\right)_{i_{D+1},\ldots,i_{2D},i_1,\ldots,i_D}$.
\end{property}
\begin{property}
$\boldsymbol{\mathcal{K}}_{\boldsymbol{\mathcal{X}}+\boldsymbol{\mathcal{Y}},\boldsymbol{\mathcal{Z}}}=\boldsymbol{\mathcal{K}}_{\boldsymbol{\mathcal{X}},\boldsymbol{\mathcal{Z}}}+\boldsymbol{\mathcal{K}}_{\boldsymbol{\mathcal{Y}},\boldsymbol{\mathcal{Z}}}$.
\end{property}
\begin{property}
If $\boldsymbol{\mathcal{X}}$ and $\boldsymbol{\mathcal{Y}}$ are independent, then $\boldsymbol{\mathcal{K}}_{\boldsymbol{\mathcal{X}},\boldsymbol{\mathcal{Y}}}=\boldsymbol{\mathcal{O}}$, where $\boldsymbol{\mathcal{O}}$ is the zero tensor.
\end{property}

The covariance tensor of a random tensor $\boldsymbol{\mathcal{X}}$ will be denoted by $\boldsymbol{\mathcal{K}}_{\boldsymbol{\mathcal{X}}\boldsymbol{\mathcal{X}}}$.
\begin{defi}
Let $\boldsymbol{\mathcal{X}}$ is a random order-D tensor with dimensional lengths $n_1 \times n_2 \times\ldots\times n_D=\mathbf{n}$, then the covariance tensor is given by
\[
\boldsymbol{\mathcal{K}}_{\boldsymbol{\mathcal{X}}\boldsymbol{\mathcal{X}}}=\mathbb{E}\left[\left(\boldsymbol{\mathcal{X}}-\mathbb{E}\left[\boldsymbol{\mathcal{X}}\right]\right)\otimes_{outer}\left(\boldsymbol{\mathcal{X}}-\mathbb{E}\left[\boldsymbol{\mathcal{X}}\right]\right)\right],
\]
where $\otimes_{outer}$ is the outer product of tensors. In other words,
\[
\left(\boldsymbol{\mathcal{K}}_{\boldsymbol{\mathcal{X}}\boldsymbol{\mathcal{X}}}\right)_{i_1,\ldots,i_D,i_{D+1},\ldots,i_{2D}}=\operatorname{cov}\left(\boldsymbol{\mathcal{X}}_{i_1,\ldots,i_D},\boldsymbol{\mathcal{X}}_{i_{D+1},\ldots,i_{2D}}\right).
\]
\end{defi}
 Thus, if $\boldsymbol{\mathcal{X}}$ is a random order-D tensor with dimensional lengths $n_1 \times n_2 \times\ldots\times n_D=\mathbf{n}$, then the covariance tensor is the order-2D tensor,  with dimensional lengths $\mathbf{n}\times\mathbf{n}$.

The covariance tensor of random tensors has the following properties. The proof of the properties is trivial, so we omit it.
\begin{property}
The covariance tensor is symmetric, i.e. $\boldsymbol{\mathcal{K}}_{\boldsymbol{\mathcal{X}},\boldsymbol{\mathcal{X}}}^{\mathrm{T}}=\boldsymbol{\mathcal{K}}_{\boldsymbol{\mathcal{X}},\boldsymbol{\mathcal{X}}}$. In other words, 
\[
\left(\boldsymbol{\mathcal{K}}_{\boldsymbol{\mathcal{X}},\boldsymbol{\mathcal{X}}}\right)_{i_1,\ldots,i_D,i_{D+1},\ldots,i_{2D}}=\left(\boldsymbol{\mathcal{K}}_{\boldsymbol{\mathcal{X}},\boldsymbol{\mathcal{X}}}\right)_{i_{D+1},\ldots,i_{2D},i_1,\ldots,i_D}.
\]
\end{property}
\begin{property}
$\boldsymbol{\mathcal{K}}_{\boldsymbol{\mathcal{X}},\boldsymbol{\mathcal{X}}}=\mathbb{E}\left[\boldsymbol{\mathcal{X}}\otimes_{outer}\boldsymbol{\mathcal{X}}\right]-\mathbb{E}\left[\boldsymbol{\mathcal{X}}\right]\otimes_{outer}\mathbb{E}\left[\boldsymbol{\mathcal{X}}\right]$.
\end{property}
\begin{property}
Let $\boldsymbol{\mathcal{X}}$ is a random order-D tensor and $\boldsymbol{\mathcal{Y}}$ is a random order-D tensor with the same dimension as $\boldsymbol{\mathcal{X}}$, then
\[
\boldsymbol{\mathcal{K}}_{\boldsymbol{\mathcal{X}}+\boldsymbol{\mathcal{Y}},\boldsymbol{\mathcal{X}}+\boldsymbol{\mathcal{Y}}}=\boldsymbol{\mathcal{K}}_{\boldsymbol{\mathcal{X}},\boldsymbol{\mathcal{X}}}+\boldsymbol{\mathcal{K}}_{\boldsymbol{\mathcal{X}},\boldsymbol{\mathcal{Y}}}+\boldsymbol{\mathcal{K}}_{\boldsymbol{\mathcal{Y}},\boldsymbol{\mathcal{X}}}+\boldsymbol{\mathcal{K}}_{\boldsymbol{\mathcal{Y}},\boldsymbol{\mathcal{Y}}},
\]
where $\boldsymbol{\mathcal{K}}_{\boldsymbol{\mathcal{X}},\boldsymbol{\mathcal{Y}}}$ and $\boldsymbol{\mathcal{K}}_{\boldsymbol{\mathcal{Y}},\boldsymbol{\mathcal{X}}}$ are the cross-covariance matrices.
\end{property}
\begin{property}
\label{kxxmat}
Let $\boldsymbol{\mathcal{X}}$ is a random tensor, then
\[
\operatorname{mat}\left(\boldsymbol{\mathcal{K}}_{\boldsymbol{\mathcal{X}}\boldsymbol{\mathcal{X}}}\right)=\mathbf{K}_{\boldsymbol{\mathcal{X}}\boldsymbol{\mathcal{X}}},
\]
where $\mathbf{K}_{\boldsymbol{\mathcal{X}\mathcal{X}}}=\mathbb{E}\left[\operatorname{vec}\left(\boldsymbol{\mathcal{X}}-\mathbb{E}\left[\boldsymbol{\mathcal{X}}\right]\right)\operatorname{vec}\left(\boldsymbol{\mathcal{X}}-\mathbb{E}\left[\boldsymbol{\mathcal{X}}\right]\right)^{\mathrm{T}}\right]$ is the covariance matrix of the random tensor $\boldsymbol{\mathcal{X}}$.
\end{property}
\section{Correlation tensor}
In a similar way, we can define the correlation tensors.

The cross-correlation tensor of random tensors $\boldsymbol{\mathcal{X}}$ and $\boldsymbol{\mathcal{Y}}$ will be denoted by $\boldsymbol{\mathcal{R}}_{\boldsymbol{\mathcal{X}}\boldsymbol{\mathcal{Y}}}$.
\begin{defi}
Let $\boldsymbol{\mathcal{X}}$ is a random order-D tensor with dimensional lengths $n_1 \times n_2 \times\ldots\times n_D=\mathbf{n}$ and $\boldsymbol{\mathcal{Y}}$ is a random order-D tensor with dimensional lengths $m_1 \times m_2 \times\ldots\times m_D=\mathbf{m}$, then the cross-correlation tensor is given by
\[
\left(\boldsymbol{\mathcal{R}}_{\boldsymbol{\mathcal{X}}\boldsymbol{\mathcal{Y}}}\right)_{i_1,\ldots,i_D,i_{D+1},\ldots,i_{2D}}=\operatorname{corr}\left(\boldsymbol{\mathcal{X}}_{i_1,\ldots,i_D},\boldsymbol{\mathcal{Y}}_{i_{D+1},\ldots,i_{2D}}\right).
\]
\end{defi}
 Thus, if $\boldsymbol{\mathcal{X}}$ is a random order-D tensor with dimensional lengths $n_1 \times n_2 \times\ldots\times n_D=\mathbf{n}$ and $\boldsymbol{\mathcal{Y}}$ is a random order-D tensor with dimensional lengths $m_1 \times m_2 \times\ldots\times m_D=\mathbf{m}$, then the cross-correlation tensor is the order-2D tensor,  with dimensional lengths $\mathbf{n}\times\mathbf{m}$.

Equivalently, the cross-correlation tensor can be seen as the cross-covariance tensor of the standardized random variables $\boldsymbol{\mathcal{X}}_{i_1,\ldots,i_D}/\sigma(\boldsymbol{\mathcal{X}}_{i_1,\ldots,i_D})$ and  $\boldsymbol{\mathcal{Y}}_{i_1,\ldots,i_D}/\sigma(\boldsymbol{\mathcal{Y}}_{i_{D+1},\ldots,i_{2D}})$.

The correlation tensor of a random tensor $\boldsymbol{\mathcal{X}}$ will be denoted by $\boldsymbol{\mathcal{R}}_{\boldsymbol{\mathcal{X}}\boldsymbol{\mathcal{X}}}$.
\begin{defi}
Let $\boldsymbol{\mathcal{X}}$ is a random order-D tensor with dimensional lengths $n_1 \times n_2 \times\ldots\times n_D=\mathbf{n}$, then the correlation tensor is given by
\[
\left(\boldsymbol{\mathcal{R}}_{\boldsymbol{\mathcal{X}}\boldsymbol{\mathcal{X}}}\right)_{i_1,\ldots,i_D,i_{D+1},\ldots,i_{2D}}=\operatorname{corr}\left(\boldsymbol{\mathcal{X}}_{i_1,\ldots,i_D},\boldsymbol{\mathcal{X}}_{i_{D+1},\ldots,i_{2D}}\right).
\]
\end{defi}
 Thus, if $\boldsymbol{\mathcal{X}}$ is a random order-D tensor with dimensional lengths $n_1 \times n_2 \times\ldots\times n_D=\mathbf{n}$, then the correlation tensor is the order-2D tensor,  with dimensional lengths $\mathbf{n}\times\mathbf{n}$.
\section{Covariance tensor as a parameter of tensor elliptical distributions}
The tensor normal distribution of a  real random order-D tensor $\boldsymbol{\mathcal{X}}$, with dimensional lengths $n_1 \times n_2 \times\ldots\times n_D=\mathbf{n}$, i.e. $\boldsymbol{\mathcal{X}}\in\mathbb{R}^{\mathbf{n}}$, can be written in the following notation:
\[
\mathcal{TN}_{\mathbf{n}}\left(\boldsymbol{\mathcal{M}},\boldsymbol{\mathcal{S}}\right),
\]
where $\boldsymbol{\mathcal{M}}\in\mathbb{R}^{\mathbf{n}}$ is a location tensor and $\boldsymbol{\mathcal{S}}\in\mathbb{R}^{\mathbf{n}\times\mathbf{n}}$ is a scale symmetric tensor.
\begin{defi}
\label{norm}
Let $\boldsymbol{\mathcal{X}}$ is a random order-D tensor, with dimensional lengths $n_1 \times n_2 \times\ldots\times n_D=\mathbf{n}$, then 
\[
\boldsymbol{\mathcal{X}}\sim\mathcal{TN}_{\mathbf{n}}\left(\boldsymbol{\mathcal{M}},\boldsymbol{\mathcal{S}}\right)\Leftrightarrow\operatorname{vec}\left(\boldsymbol{\mathcal{X}}\right)\sim\mathcal{N}_{n^{*}}\left(\operatorname{vec}\left(\boldsymbol{\mathcal{M}}\right),\operatorname{mat}\left(\boldsymbol{\mathcal{S}}\right)\right),
\]
where $n^{*}=\prod_{i=1}^{D}n_i$. In other words, the probability density function of the random tensor is given by
\[
f(\boldsymbol{\mathcal{X}})=\frac{\operatorname{exp}\left(-\frac{1}{2}\operatorname{vec}\left(\boldsymbol{\mathcal{X}}-\boldsymbol{\mathcal{M}}\right)^{\mathrm{T}}\operatorname{mat}\left(\boldsymbol{\mathcal{S}}\right)^{-1}\operatorname{vec}\left(\boldsymbol{\mathcal{X}}-\boldsymbol{\mathcal{M}}\right)\right)}{(2\pi)^{n^{*}/2}\det\left(\operatorname{mat}\left(\boldsymbol{\mathcal{S}}\right)\right)^{1/2}}.
\]
\end{defi}
\begin{theorem}
Let $\boldsymbol{\mathcal{X}}$ is a random order-D tensor, with dimensional lengths $n_1 \times n_2 \times\ldots\times n_D=\mathbf{n}$, and $\boldsymbol{\mathcal{X}}\sim\mathcal{TN}_{\mathbf{n}}\left(\boldsymbol{\mathcal{M}},\boldsymbol{\mathcal{S}}\right)$, then the expected value of the random tensor is equal to $\boldsymbol{\mathcal{M}}$.
\end{theorem}
\begin{proof}
By Definition \ref{norm},
\[
\boldsymbol{\mathcal{X}}\sim\mathcal{TN}_{\mathbf{n}}\left(\boldsymbol{\mathcal{M}},\boldsymbol{\mathcal{S}}\right)\Leftrightarrow\operatorname{vec}\left(\boldsymbol{\mathcal{X}}\right)\sim\mathcal{N}_{n^{*}}\left(\operatorname{vec}\left(\boldsymbol{\mathcal{M}}\right),\operatorname{mat}\left(\boldsymbol{\mathcal{S}}\right)\right).
\]
Therefore, $\mathbb{E}\left[\operatorname{vec}\left(\boldsymbol{\mathcal{X}}\right)\right]=\operatorname{vec}\left(\boldsymbol{\mathcal{M}}\right)$. Thus, $\mathbb{E}\left[\boldsymbol{\mathcal{X}}\right]=\boldsymbol{\mathcal{M}}$.
\end{proof}
\begin{theorem}
Let $\boldsymbol{\mathcal{X}}$ is a random order-D tensor, with dimensional lengths $n_1 \times n_2 \times\ldots\times n_D=\mathbf{n}$, and $\boldsymbol{\mathcal{X}}\sim\mathcal{TN}_{\mathbf{n}}\left(\boldsymbol{\mathcal{M}},\boldsymbol{\mathcal{S}}\right)$, then the covariance tensor of the random tensor is equal to $\boldsymbol{\mathcal{S}}$.
\end{theorem}
\begin{proof}
By Definition \ref{norm},
\[
\boldsymbol{\mathcal{X}}\sim\mathcal{TN}_{\mathbf{n}}\left(\boldsymbol{\mathcal{M}},\boldsymbol{\mathcal{S}}\right)\Leftrightarrow\operatorname{vec}\left(\boldsymbol{\mathcal{X}}\right)\sim\mathcal{N}_{n^{*}}\left(\operatorname{vec}\left(\boldsymbol{\mathcal{M}}\right),\operatorname{mat}\left(\boldsymbol{\mathcal{S}}\right)\right).
\]
Therefore, $\mathbf{K}_{\operatorname{vec}\left(\boldsymbol{\mathcal{X}}\right),\operatorname{vec}\left(\boldsymbol{\mathcal{X}}\right)}=\mathbf{K}_{\boldsymbol{\mathcal{X}\mathcal{X}}}=\operatorname{mat}\left(\boldsymbol{\mathcal{S}}\right)$.

By Property \ref{kxxmat},
\[
\operatorname{mat}\left(\boldsymbol{\mathcal{K}}_{\boldsymbol{\mathcal{X}}\boldsymbol{\mathcal{X}}}\right)=\operatorname{mat}\left(\boldsymbol{\mathcal{S}}\right).
\]
Thus, $\boldsymbol{\mathcal{K}}_{\boldsymbol{\mathcal{X}}\boldsymbol{\mathcal{X}}}=\boldsymbol{\mathcal{S}}$.
\end{proof}
\begin{theorem}
Let $\boldsymbol{\mathcal{X}}=\left(\boldsymbol{\mathcal{X}}_{i_1,\ldots,i_D}\right)$ is a random order-D tensor, with dimensional lengths $n_1 \times n_2 \times\ldots\times n_D=\mathbf{n}$, and $\boldsymbol{\mathcal{X}}\sim\mathcal{TN}_{\mathbf{n}}\left(\boldsymbol{\mathcal{M}},\boldsymbol{\mathcal{S}}\right)$, then the probability density function of the random tensor is given by
\[
f\left(\boldsymbol{\mathcal{X}}\right)=\frac{\operatorname{exp}\left(-\frac{1}{2}\left(\boldsymbol{\mathcal{X}}-\boldsymbol{\mathcal{M}}\right):\boldsymbol{\mathcal{S}}^{-1}:\left(\boldsymbol{\mathcal{X}}-\boldsymbol{\mathcal{M}}\right)\right)}{(2\pi)^{n^{*}/2}\det\left(\boldsymbol{\mathcal{S}}\right)^{1/2}},
\]
where $n^{*}=\prod_{i=1}^{D}n_i$ and $:$ is the double dot product of tensors, i.e
\[
\left(\boldsymbol{\mathcal{X}}-\boldsymbol{\mathcal{M}}\right):\boldsymbol{\mathcal{S}}^{-1}:\left(\boldsymbol{\mathcal{X}}-\boldsymbol{\mathcal{M}}\right)=\left(\boldsymbol{\mathcal{X}}-\boldsymbol{\mathcal{M}}\right)_{i_1,\ldots,i_D}\left(\boldsymbol{\mathcal{S}}^{-1}\right)^{i_1,\ldots,i_D}_{i_{D+1},\ldots,i_{2D}}\left(\boldsymbol{\mathcal{X}}-\boldsymbol{\mathcal{M}}\right)^{i_{D+1},\ldots,i_{2D}}.
\]
\end{theorem}
\begin{proof}
By Definition \ref{norm}, $\boldsymbol{\mathcal{X}}\sim\mathcal{TN}_{\mathbf{n}}\left(\boldsymbol{\mathcal{M}},\boldsymbol{\mathcal{S}}\right)$ if and only if the probability density function of the tensor is given by
\[
f(\boldsymbol{\mathcal{X}})=\frac{\operatorname{exp}\left(-\frac{1}{2}\operatorname{vec}\left(\boldsymbol{\mathcal{X}}-\boldsymbol{\mathcal{M}}\right)^{\mathrm{T}}\operatorname{mat}\left(\boldsymbol{\mathcal{S}}\right)^{-1}\operatorname{vec}\left(\boldsymbol{\mathcal{X}}-\boldsymbol{\mathcal{M}}\right)\right)}{(2\pi)^{n^{*}/2}\det\left(\operatorname{mat}\left(\boldsymbol{\mathcal{S}}\right)\right)^{1/2}}.
\]
By Definition \ref{det}, $\det\left(\boldsymbol{\mathcal{X}}\right)=\det\left(\operatorname{mat}\left(\boldsymbol{\mathcal{X}}\right)\right)$. By Property \ref{mat-1}, $\operatorname{mat}\left(\boldsymbol{\mathcal{S}}^{-1}\right)=\operatorname{mat}\left(\boldsymbol{\mathcal{S}}\right)^{-1}$.

Therefore,
\[
f(\boldsymbol{\mathcal{X}})=\frac{\operatorname{exp}\left(-\frac{1}{2}\operatorname{vec}\left(\boldsymbol{\mathcal{X}}-\boldsymbol{\mathcal{M}}\right)^{\mathrm{T}}\operatorname{mat}\left(\boldsymbol{\mathcal{S}}^{-1}\right)\operatorname{vec}\left(\boldsymbol{\mathcal{X}}-\boldsymbol{\mathcal{M}}\right)\right)}{(2\pi)^{n^{*}/2}\det\left(\boldsymbol{\mathcal{S}}\right)^{1/2}}.
\]
Obviously, 
\[
\operatorname{vec}\left(\boldsymbol{\mathcal{X}}-\boldsymbol{\mathcal{M}}\right)^{\mathrm{T}}\operatorname{mat}\left(\boldsymbol{\mathcal{S}}^{-1}\right)\operatorname{vec}\left(\boldsymbol{\mathcal{X}}-\boldsymbol{\mathcal{M}}\right)=\left(\boldsymbol{\mathcal{X}}-\boldsymbol{\mathcal{M}}\right)_{i_1,\ldots,i_D}\left(\boldsymbol{\mathcal{S}}^{-1}\right)^{i_1,\ldots,i_D}_{i_{D+1},\ldots,i_{2D}}\left(\boldsymbol{\mathcal{X}}-\boldsymbol{\mathcal{M}}\right)^{i_{D+1},\ldots,i_{2D}}.
\]
Thus, 
\[
f\left(\boldsymbol{\mathcal{X}}\right)=\frac{\operatorname{exp}\left(-\frac{1}{2}\left(\boldsymbol{\mathcal{X}}-\boldsymbol{\mathcal{M}}\right):\boldsymbol{\mathcal{S}}^{-1}:\left(\boldsymbol{\mathcal{X}}-\boldsymbol{\mathcal{M}}\right)\right)}{(2\pi)^{n^{*}/2}\det\left(\boldsymbol{\mathcal{S}}\right)^{1/2}}.
\]
\end{proof}
Thus, if $\boldsymbol{\mathcal{X}}$ is normally distributed, or more generally elliptically distributed, then its probability density function can be expressed in terms of the covariance tensor.
\begin{defi}
If $\boldsymbol{\mathcal{X}}$ is a random order-D tensor, with dimensional lengths $n_1 \times n_2 \times\ldots\times n_D=\mathbf{n}$, and the random tensor is elliptically distributed, then the probability density function of the random tensor is given by
\[
f\left(\boldsymbol{\mathcal{X}}\right)=c_{n^{*}}g\left(\left(\boldsymbol{\mathcal{X}}-\boldsymbol{\mathcal{M}}\right):\boldsymbol{\mathcal{S}}^{-1}:\left(\boldsymbol{\mathcal{X}}-\boldsymbol{\mathcal{M}}\right)\right),
\]
where $c_{n^{*}}$ is the normalizing constant, $\boldsymbol{\mathcal{M}}$ is a location tensor and $\boldsymbol{\mathcal{S}}$ is a symmetric scale tensor which is proportional to the covariance tensor.
\end{defi}
\section{Relationship to Arashi's representation of tensor distributions}
\begin{theorem}
Let $\boldsymbol{\mathcal{X}}$ is a random order-D tensor, with dimensional lengths $n_1 \times n_2 \times\ldots\times n_D=\mathbf{n}$, and $\operatorname{mat}\left(\boldsymbol{\mathcal{S}}\right)=\bigotimes_{i=1}^D\mathbf{\Sigma}_{i}$, then
\[
\boldsymbol{\mathcal{X}}\sim\mathcal{TN}_{\mathbf{n}}\left(\boldsymbol{\mathcal{M}},\boldsymbol{\mathcal{S}}\right)\Leftrightarrow\boldsymbol{\mathcal{X}}\sim\mathcal{TN}_{\mathbf{n}}\left(\boldsymbol{\mathcal{M}},\mathbf{\Sigma}_1,\ldots,\mathbf{\Sigma}_D\right),
\]
where $\otimes$ is the Kronecker product of matrices.
\end{theorem}
\begin{proof}
By Definition \ref{norm},
\[
\boldsymbol{\mathcal{X}}\sim\mathcal{TN}_{\mathbf{n}}\left(\boldsymbol{\mathcal{M}},\boldsymbol{\mathcal{S}}\right)\Leftrightarrow\operatorname{vec}\left(\boldsymbol{\mathcal{X}}\right)\sim\mathcal{N}_{n^{*}}\left(\operatorname{vec}\left(\boldsymbol{\mathcal{M}}\right),\operatorname{mat}\left(\boldsymbol{\mathcal{S}}\right)\right).
\]
By Arashi's definition \cite{ara},
\[
\boldsymbol{\mathcal{X}}\sim\mathcal{TN}_{\mathbf{n}}\left(\boldsymbol{\mathcal{M}},\mathbf{\Sigma}_1,\ldots,\mathbf{\Sigma}_D\right)\Leftrightarrow\operatorname{vec}\left(\boldsymbol{\mathcal{X}}\right)\sim\mathcal{N}_{n^{*}}\left(\operatorname{vec}\left(\boldsymbol{\mathcal{M}}\right),\bigotimes_{i=1}^D\mathbf{\Sigma}_{i}\right).
\]
Therefore,
\[
\operatorname{vec}\left(\boldsymbol{\mathcal{X}}\right)\sim\mathcal{N}_{n^{*}}\left(\operatorname{vec}\left(\boldsymbol{\mathcal{M}}\right),\operatorname{mat}\left(\boldsymbol{\mathcal{S}}\right)\right)\Leftrightarrow\operatorname{vec}\left(\boldsymbol{\mathcal{X}}\right)\sim\mathcal{N}_{n^{*}}\left(\operatorname{vec}\left(\boldsymbol{\mathcal{M}}\right),\bigotimes_{i=1}^D\mathbf{\Sigma}_{i}\right).
\]
Thus,
\[
\boldsymbol{\mathcal{X}}\sim\mathcal{TN}_{\mathbf{n}}\left(\boldsymbol{\mathcal{M}},\boldsymbol{\mathcal{S}}\right)\Leftrightarrow\boldsymbol{\mathcal{X}}\sim\mathcal{TN}_{\mathbf{n}}\left(\boldsymbol{\mathcal{M}},\mathbf{\Sigma}_1,\ldots,\mathbf{\Sigma}_D\right).
\]
\end{proof}
In a similar way, we can prove the equivalence of the tensor elliptical distribution representations.
\section{Conclusions}
In this article, we define a matrix representation of a tensor, tensor matricization. After that, we presented some properties of the matricization. We also defined the determinant of tensor and proved properties of the determinant. The determinant of a tensor has the same properties as the determinant of a matrix.

In this article, we defined the covariance tensor. As we wrote earlier, if a random tensor is order-D tensor, then the covariance tensor of the random tensor is order-2D tensor. Indeed the covariance tensor of a scalar, i.e order-0 tensor, is scalar, i.e order-0 tensor. The covariance tensor of a vector, i.e order-1 tensor, is matrix, i.e order-2 tensor.  The covariance tensor has the same properties as the covariance matrix.

Based on all of the above, we can argue that the covariance tensor is a generalization of the variance for random tensors.

In a similar way, we have defined the correlation tensor. As we wrote earlier, if a random tensor is order-D tensor, then the correlation tensor of the random tensor is order-2D tensor. Indeed the correlation tensor of a scalar, i.e order-0 tensor, is scalar, i.e order-0 tensor. The correlation tensor of a vector, i.e order-1 tensor, is matrix, i.e order-2 tensor.

Based on all of the above, we can argue that the correlation tensor is a generalization of the correlation for random tensors.

The covariance and correlation tensors can be used to characterize random tensors that have specific distributions. These tensors can be applied to analyze the tensor distributions described in \cite{ara}, \cite{gal} and \cite{yur}.

In this article, we also defined the elliptical distributions. We proved the equivalence of the tensor elliptical distribution representations.
Also, from the equivalence of representations it follows that our definition of the tensor determinant is correct.

\end{document}